\documentclass[a4paper, leqno, oneside]{amsart}

\usepackage[latin1]{inputenc}
\usepackage{amsmath}
\usepackage{amsfonts}
\usepackage{amssymb}
\usepackage{graphicx}

\newtheorem{theorem}{Theorem}[]
\newtheorem{lemma}{Lemma}[]
\theoremstyle{plain}
\newtheorem*{proposition}{Proposition}

\begin{document}

\author{J. Pascal Gollin and Karl Heuer}
\address{J. Pascal Gollin, Fachbereich Mathematik, Universit\"{a}t Hamburg, Bundesstra{\ss}e 55, 20146 Hamburg, Germany}
\address{Karl Heuer, Fachbereich Mathematik, Universit\"{a}t Hamburg, Bundesstra{\ss}e 55, 20146 Hamburg, Germany}

\title[]{Infinite end-devouring sets of rays with prescribed start vertices}

\begin{abstract}
We prove that every end of a graph contains either uncountably many disjoint rays or a set of disjoint rays that meet all rays of the end and start at any prescribed feasible set of start vertices.
This confirms a conjecture of Georgakopoulos.
\end{abstract}
\maketitle

\section{Introduction}

Looking for spanning structures in infinite graphs such as spanning trees or Hamilton cycles often involves difficulties that are not present when one considers finite graphs.
It turned out that the concept of \textit{ends} of an infinite graph is crucial for questions dealing with such structures. Especially for \textit{locally finite graphs}, i.e., graphs in which every vertex has finite degree, ends allow us to define these objects in a more general topological setting~\cite{diestel_arx}.

Nevertheless, the definition of an end of a graph is purely combinatorial:
We call one-way infinite paths \textit{rays} and the vertex of degree~$1$ in them the \textit{start vertex} of the ray.
For any graph $G$ we call two rays \textit{equivalent} in $G$ if they cannot be separated by finitely many vertices.
It is easy to check that this defines an equivalence relation on the set of all rays in the graph $G$.
The equivalence classes of this relation are called the \textit{ends} of $G$ and a ray contained in an end $\omega$ of $G$ is referred to as an $\omega$-\textit{ray}.

When we focus on the structure of ends of an infinite graph $G$, we observe that \textit{normal spanning trees} of $G$, i.e., rooted spanning trees of $G$ such that the endvertices of every edge of $G$ are comparable in the induced tree-order, have a powerful property:
For any normal spanning tree $T$ of $G$ and every end $\omega$ of $G$ there is a unique $\omega$-ray in $T$ which starts at the root of $T$ and has the property that it meets every $\omega$-ray of $G$, see~\cite[Sect.~8.2]{diestel_buch}.
For any graph $G$, we say that an $\omega$-ray with this property \textit{devours} the end $\omega$ of $G$.
Similarly, we say that a set of $\omega$-rays \textit{devours} $\omega$ if every $\omega$-ray in $G$ meets at least one ray out of the set.

End-devouring sets of rays are helpful for the construction of spanning structures such as infinite Hamilton cycles, as done by Georgakopoulos~\cite{agelos-HC}.
He proved the following proposition about the existence of finite sets of rays devouring any \textit{countable end}, i.e., an end which does not contain uncountably many disjoint rays.
Note that the property of an end being countable is equivalent to the existence of a finite or countably infinite set of rays devouring the end.

\begin{proposition}\label{k_dev_rays}\cite{agelos-HC}
Let $G$ be a graph and $\omega$ be a countable end of $G$.
If $G$ has a set $\mathcal{R}$ of $k \in \mathbb{N}$ disjoint $\omega$-rays, then it also has a set $\mathcal{R}'$ of $k$ disjoint $\omega$-rays that devours~$\omega$.
Moreover, $\mathcal{R}'$ can be chosen so that its rays have the same start vertices as the rays in $\mathcal{R}$.
\end{proposition}

For the proof of this proposition Georgakopoulos recursively applies a construction similar to the one yielding normal spanning trees to find rays for the end-devouring set.
However, this proof strategy does not suffice to give a version of this proposition for infinitely many rays.
He conjectured that such a version remains true~\cite[Problem~1]{agelos-HC}.
We confirm this conjecture with the following theorem, which also covers the proposition above.

\begin{theorem}\label{main-Thm}
Let $G$ be a graph, $\omega$ a countable end of $G$ and $\mathcal{R}$ any set of disjoint $\omega$-rays.
Then there exists a set $\mathcal{R}'$ of disjoint $\omega$-rays that devours~$\omega$ and the start vertices of the rays in $\mathcal{R}$ and $\mathcal{R}'$ are the same.
\end{theorem}

Note that, in contrast to the proposition, the difficulty of Theorem~\ref{main-Thm} for an infinite set $\mathcal{R}$ comes from fixing the set of start vertices, since any inclusion-maximal set of disjoint $\omega$-rays devours $\omega$.

Moreover, we shall sketch at the beginning of Section~3 an example of why it is necessary to require the end to be countable.

\section{Preliminaries}

All graphs in this paper are simple and undirected.
For basic facts about finite and infinite graphs we refer the reader to \cite{diestel_buch}.
If not stated differently, we also use the notation of \cite{diestel_buch}.

For two graphs $G$ and $H$ we define their Cartesian product $G \times H$ as the graph on the vertex set $V(G) \times V(H)$ such that two vertices ${(g_1, h_1), (g_2, h_2) \in V(G\times H)}$ are adjacent if and only if either $h_1 = h_2$ and $g_1g_2 \in E(G)$ or $g_1 = g_2$ and $h_1h_2 \in E(H)$ holds.
If $H$ consists just of a single vertex $h$, we may write $G \times h$ instead of $G \times H$.
We define the union $G \cup H$ of $G$ and $H$ as the graph $(V(G) \cup V(H), E(G) \cup E(H))$.

Any ray $T$ that is a subgraph of a ray $R$ is called a \textit{tail} of $R$.
For a vertex $v$ and an end $\omega$ of a graph $G$ we say that a vertex set $X \subseteq V(G)$ \textit{separates} $v$ from $\omega$ if there does not exist any $\omega$-ray that is disjoint from $X$ and contains $v$.

For a finite set $M$ of vertices of a graph $G$ and an end $\omega$ of $G$, let $C(M, \omega)$ denote the unique component of $G - M$ that contains a tail of every $\omega$-ray.

Given a path or ray $Q$ containing two vertices $v$ and $w$ we denote the unique $v \,$--$\, w$ path in $Q$ by $vQw$.
Furthermore, for $Q$ being a $v \,$--$\, w$ path we write $v \bar{Q}$ for the path that is obtained from $Q$ by deleting~$w$.

For a ray $R$ that contains a vertex $v$ we write $vR$ for the tail of $R$ with start vertex $v$.

We use the following notion to abbreviate concatenations of paths and rays.
Let~$P$ be a $v \,$--$\, w$ path for two vertices $v$ and $w$, and let $Q$ be either a ray or another path such that $V(P) \cap V(Q) = \lbrace w \rbrace$.
Then we write $PQ$ for the path or ray $P \cup Q$, respectively.
We omit writing brackets when stating concatenations of more than two paths or rays.

\section{Theorem}
Before we prove Theorem~\ref{main-Thm}, we show the necessity of assuming the considered end to be countable.
Let $R$ denote a ray and $K_{1, \aleph_1}$ the star with center $c$ and $\aleph_1$ many leaves~$\ell_\alpha$ for all ordinals $\alpha < \aleph_1$.
Now we define $G = R \times K_{1, \aleph_1}$.
By definition,~$G$ contains uncountably many disjoint rays. Furthermore, all rays in~$G$ belong to a unique end $\omega$.
For reasons of cardinality, a countable set of rays devouring $\omega$ cannot exist in~$G$.
But not even for all uncountable sets $\mathcal{R}$ of disjoint $\omega$-rays we can find a set $\mathcal{R}'$ of $\omega$-rays that devours $\omega$  such that the start vertices of the rays of~$\mathcal{R}$ and~$\mathcal{R}'$ coincide.
Take any partition $X \cup Y = \aleph_1$ into uncountable sets $X$ and $Y$.
Let $\mathcal{R} := \{ R \times \ell_\beta \; ; \; \beta \in X\}$, which then prescribes the set of start vertices.
Suppose there is a set $\mathcal{R}'$ of disjoint rays with the same start vertices devouring $\omega$.
For every $\gamma \in Y$ there needs to be a ray in $\mathcal{R}'$ containing a vertex of $R \times \ell_\gamma$.
By definition of~$G$ this ray has to contain a vertex of~$R \times c$.
Hence in order to hit every ray $R \times \ell_\gamma$ for $\gamma \in Y$ we need uncountably many rays of $\mathcal{R}'$ each hitting the countable vertex set of $R \times c$.
This is a contradiction.

\pagebreak
For the proof of Theorem~\ref{main-Thm} we shall use the following characterisation of $\omega$-rays.

\setcounter{lemma}{1}
\begin{lemma}\label{lemma:omega-rays}
    Let $G$ be a graph, $\omega$ an end of $G$ and $\mathcal{R}_{max}$ an inclusion-maximal set of pairwise disjoint $\omega$-rays.
    A ray $R \subseteq G$ is an $\omega$-ray, if and only if it meets rays of~$\mathcal{R}_{max}$ infinitely often.
\end{lemma}

\begin{proof}
    Let $W$ denote the set $\bigcup \{ V(R)\, ;\, R \in \mathcal{R}_{max}\}$.
    
    If $R$ is an $\omega$-ray, then each tail of $R$ meets a ray of $\mathcal{R}_{max}$ since $\mathcal{R}_{max}$ is inclusion-maximal.
    Hence $R$ meets $W$ infinitely often.
    
    Suppose for a contradiction that $R$ is an $\omega'$-ray for an end $\omega' \neq \omega$ of~$G$ that contains infinitely many vertices of $W$.
    Let $M$ be a finite set of vertices such that the two components $C := C(M, \omega)$ and $C' := C(M, \omega')$ of $G-M$ are different.
    By the pigeonhole principle there is either one $\omega$-ray of~$\mathcal{R}_{max}$ containing infinitely many vertices of~${V(C') \cap V(R) \cap W}$, 
    or infinitely many disjoint rays of~$\mathcal{R}_{max}$ containing those vertices.
    In both cases we get an $\omega$-ray with a tail in $C'$, since we cannot leave $C'$ infinitely often through the finite set $M$.
    But this contradicts the definition of~$C$.
\end{proof}

The idea of the proof of Theorem~\ref{main-Thm} is to build the end-devouring set of rays in countably many steps while fixing in each step only finitely many finite paths as initial segments instead of whole rays, but ensuring that they can be extended to rays.
In order to guarantee that the set of rays turns out to devour the end, we also fix an inclusion maximal set of vertex disjoint rays of our specific end, so a countable set, and an enumeration of the vertices on these rays.
Then we try in each step to either contain or separate the least vertex with respect to the enumeration that is not already dealt with from the end with appropriately chosen initial segments if possible.
Otherwise, we extend a finite number of initial segments while still ensuring that all initial segments can be extended to rays.
Although it is impossible to always contain or separate the considered vertex with our initial segments while being able to continue with the construction, it will turn out that the rays we obtain as the union of all initial segments actually do this.

\begin{proof}[Proof of Theorem~\ref{main-Thm}]
Let us fix a finite or infinite enumeration $\{ R_j \; ; \; j < \lvert \mathcal{R} \rvert \}$ of the rays in $\mathcal{R}$.
Furthermore, let $s_j$ denote the start vertex of~$R_j$ for every $j < |\mathcal{R}|$ and define $S := \{ s_j \; ; \; j < \lvert \mathcal{R} \rvert \}$.

Next we fix an inclusion-maximal set $\mathcal{R}_{max}$ of pairwise disjoint $\omega$-rays and an enumeration $\{ v_i  \; ; \; i \in \mathbb{N} \}$ of the vertices in $W:= \bigcup \lbrace V(R) \; ; \; R \in \mathcal{R}_{max} \rbrace$.
This is possible since $\omega$ is countable by assumption.

We do an inductive construction such that the following holds for every~${i \in \mathbb{N}}$:
\begin{enumerate}
\item $P^i_{s}$ is a path with start vertex $s$ for every $s \in S$.
\item $P_s^{i} = s$ for all but finitely many $s \in S$.
\item $P^{i-1}_{s} \subseteq P^{i}_{s}$ for every $s \in S$.
\item For every ${s = s_j \in S}$ with ${j < \min\{i, |\mathcal{R}|\}}$ there is a~$w_{s}^{i} \in W \cap (P_{s}^i \setminus P_{s}^{i-1})$.
\item $P^i_{s}$ and $P^i_{s'}$ are disjoint for all $s, s' \in S$ with $s \neq s'$.
\item For every $s \in S$ there exists an $\omega$-ray $R^i_s$ with $P^i_{s}$ as initial segment and $s$ as start vertex such that all rays $R^i_s$ are pairwise disjoint.
\end{enumerate}
If possible and not spoiling any of the properties (1) to (6), we incorporate the following property:
\begin{enumerate}
\item[$(\ast)$] $\bigcup\limits_{s\in S}P^{i}_s$ either contains $v_{i-1}$ or separates $v_{i-1}$ from $\omega$ if $i > 0$.
\end{enumerate}

We begin the construction for $i=0$ by defining $P^0_{s} := s =: P^{-1}_{s}$ for every~$s \in S$.
All conditions are fulfilled as witnessed by $R^0_{s_j} := R_{j}$ for every $j < |\mathcal{R}|$.

Now suppose we have done the construction up to some number $i \in \mathbb{N}$.
If we can continue with the construction in step $i+1$ such that properties (1) to (6) together with $(\ast)$ hold, we do so and define all initial segments $P^{i+1}_s$ and rays $R^{i+1}_s$ accordingly. 
Otherwise, we set for all $s \in S$
\begin{align*}
P^{i+1}_{s} &:= s R^i_s w^i_s &&\textnormal{ if } s = s_j \textnormal{ for } j < \min \{ i+1, |\mathcal R| \rbrace \textnormal{ and}\\
P^{i+1}_{s} &:= P^{i}_s &&\textnormal{ otherwise},
\end{align*}
where $w_{s}^{i}$ denotes the first vertex of $W$ on $R^i_s \setminus P^{i}_s$ which exist by Lemma~\ref{lemma:omega-rays}.
With these definitions properties (1) up to (5) hold for $i+1$.
Witnessed by $R^{i+1}_s := R^{i}_s$ for every $s \in S$ we immediately satisfy (6) too.
This completes the inductive part of the construction.

Using the paths $P^i_s$ we now define the desired $\omega$-rays of $\mathcal{R}'$.
We set ${R'_s := \bigcup_{i \in \mathbb{N}}P^i_s}$ for every $s \in S$ and $\mathcal{R}' := \lbrace R'_s \; ; \; s \in S \rbrace$.
Properties~(1), (3) and (4) ensure that~$R'_s$ is a ray with start vertex $s$ for every $s \in S$, while we obtain due to property~(5) that all rays $R'_s$ are pairwise disjoint.
Property~(4) also ensures that all rays in~$\mathcal{R}'$ are~$\omega$-rays by Lemma~\ref{lemma:omega-rays}.

It remains to prove that the set $\mathcal{R}'$ devours the end $\omega$.
Suppose for a contradiction that there exists an $\omega$-ray $R$ disjoint from $\bigcup \mathcal{R}'$.
By the maximality of our chosen set of $\omega$-rays $\mathcal{R}_{max}$, we know that $R$ contains a vertex $v_{j}$ for some $j \in \mathbb{N}$.
Without loss of generality, let $v_{j}$ be the start vertex of $R$.
Let~$P$ be an $R \,$--$\, \bigcup \mathcal{R}'$ path 
among those ones that are disjoint from $\bigcup_{s \in S} s \bar{P}^{j+1}_s$ 
for which~$v_{j}Rp$ is as short as possible where~$p$ denotes the common vertex of~$P$ and~$R$.
Such a path exists, because all rays in~$\mathcal{R}' \cup \lbrace R \rbrace$ are equivalent and $\bigcup_{s \in S} s \bar{P}^{j+1}_s$ is finite by property~(2).
Let~${t \in S}$ and $q \in V(G)$ be such that $V(P) \cap V(R'_t) = \lbrace q \rbrace$.
Furthermore, let~$R^*$ be an $\omega$-ray with start vertex~$r^* \in R$ such that $R^*$ is disjoint from $\bigcup_{s \in S} R'_s$ and ${P(pRr^*) \cap R^* = \{r^*\}}$
for which~$v_{j}Rr^*$ is as short as possible.
Since~$p$ and~$pR$ are candidates for $r^*$ and~$R^*$, respectively, such a choice is possible.
Now we are able to point out a contradiction because we could have incorporated property~$(\ast)$ in the construction without violating any of the other properties in step $j+1$.
For this we define
\[
    \hat{P}^{j+1}_t := (tR'_tq) P (pRr^*) \quad \textnormal{ and } \quad 
    \hat{R}^{j+1}_t := \hat{P}^{j+1}_t  R^*\ ;
\]
and replace in step $j+1$ the ray $R^{j+1}_t$ by $\hat{R}^{j+1}_t$, the path $P^{j+1}_t$ by $\hat{P}^{j+1}_t$ and for all $s \in S \setminus \lbrace t \rbrace$ the ray $R^{j+1}_s$ by $R'_s$ while keeping $P^{j+1}_s$ as it was defined (cf.~Figure~\ref{image}).
By this construction properties~(1) to~(6) are satisfied.

\begin{figure}[htbp]
\centering
\includegraphics[width=9cm]{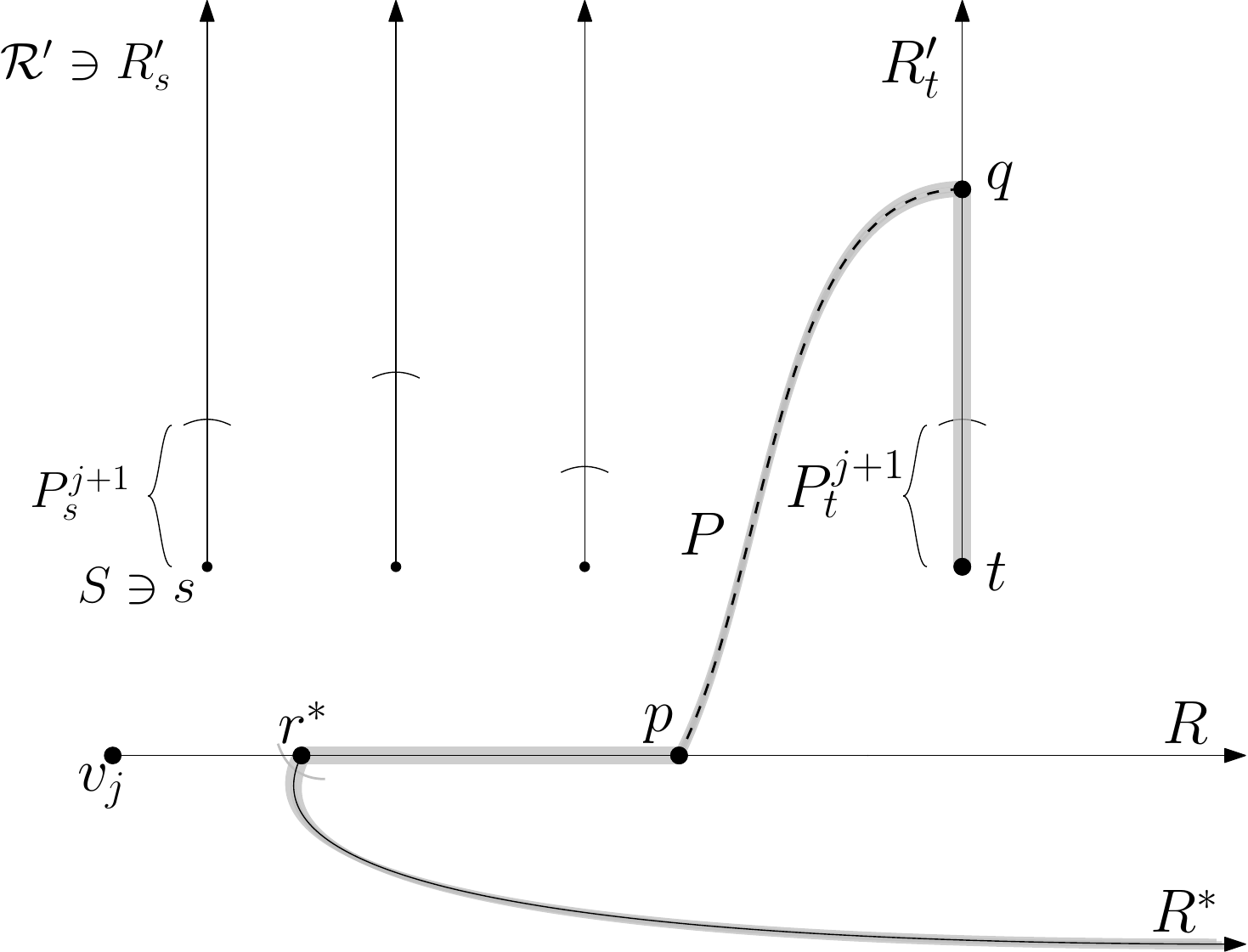}
\caption{Sketch of the situation above with $\hat{R}^{j+1}_t$ highlighted in grey.}
\label{image}
\end{figure}

Now we show that~$(\ast)$ holds as well.
Suppose for a contradiction that there exists an $\omega$-ray $Z$ disjoint from $\bigcup_{s\in S}P^{j+1}_s$ with start vertex $v_j$.
First note that $Z$ is disjoint from $r^*Rp \subseteq {P}^{j+1}_t$.
Let us now show that~$Z$ is also disjoint from ${pR \cup \bigcup_{s\in S}R'_s}$.
Otherwise, let~$z$ denote the first vertex along $Z$ that lies in ${pR \cup \bigcup_{s\in S}R'_s}$.
However, $z$ cannot be contained in $pR$, as this would contradict the choice of $r^*$, and it cannot be an element of $\bigcup_{s\in S}R'_s$ since this would contradict the choice of $p$.
But now with $Z$ being not only disjoint from $pR \cup \bigcup_{s\in S}R'_s$ but also from $r^*Rp$, we get again a contradiction to the choice of~$r^*$.
Hence, we would have been able to incorporate property $(\ast)$ without violating any of the properties (1) to (6) in step $j+1$ of our construction.
This, however, is a contradiction since we always incorporated property $(\ast)$ under the condition of maintaining properties (1) to (6).
So we arrived at a contradiction to the existence of the ray $R$ since by~$(\ast)$ every ray containing~$v_j$ meets the initial segments of rays fixed in our construction at step~$j+1$.
Therefore, the set $\mathcal{R}'$ devours the end $\omega$.
\end{proof}

\end{document}